\documentclass[twoside,11pt]{article}

\usepackage{amscd,amsmath,amsthm,amssymb}
\usepackage{pstricks,pst-node}
\usepackage{graphics}
\usepackage{latexsym,amsbsy,mathrsfs}
\usepackage[enableskew,vcentermath]{youngtab}
\usepackage{titling}
\usepackage{txfonts}

\usepackage[bbgreekl]{mathbbol}

\let\<=\langle
\let\>=\rangle
\def\N{\mathbb N}
\def\Z{\mathbb Z}
\def\C{\mathbb C}
\def\O{{\mathcal{O}}}
\def\fg{\mathfrak{g}}
\def\fn{\mathfrak{n}}
\def\fh{\mathfrak{h}}
\def\lam{\lambda}
\def\Lam{\Lambda}
\def\H{\rm H}
\def\mL{\mathbb L}
\newcommand\HH{\mathscr{H}}
\def\mL{\mathbb{L}}
\def\mD{\mathbb{\Delta}}
\def\mP{\mathbb{P}}

\newcommand{\frp}{\mathfrak{p}}
\newcommand{\frg}{\mathfrak{g}}
\newcommand{\fru}{\mathfrak{u}}
\newcommand{\frl}{\mathfrak{l}}

\DeclareMathOperator\lmod{\!-mod}
\DeclareMathOperator\rmod{mod-\!}
\DeclareMathOperator\lgmod{\!-gmod}
\DeclareMathOperator\rgmod{gmod-\!}
\DeclareMathOperator\Hom{Hom}
\DeclareMathOperator\ho{hom}
\DeclareMathOperator\End{End}
\DeclareMathOperator\Ext{Ext}

\DeclareMathOperator\id{id}
\DeclareMathOperator\res{res}
\DeclareMathOperator\pr{pr}
\DeclareMathOperator\ind{ind}
\DeclareMathOperator\tr{tr}

\numberwithin{equation}{section}
\newtheorem{theorem}{Theorem}[section]
\newtheorem{lemma}[theorem]{Lemma}
\newtheorem{corollary}[theorem]{Corollary}
\newtheorem{proposition}[theorem]{Proposition}
\theoremstyle{definition}
\newtheorem{definition}[theorem]{Definition}
\newtheorem{dfnthm}[theorem]{Definition and Theorem}

\theoremstyle{remark}
\newtheorem{remark}[theorem]{Remark}

\def\title#1#2{\thispagestyle{empty}\null\vbox to1cm{}\par
\noindent{\sffamily\bfseries\LARGE#1}\\[1cm]{\sffamily\bfseries\LARGE#2}\par
\vbox to1cm{}}
\def\author#1{\renewcommand{\thefootnote}{\fnsymbol{footnote}}\noindent{\normalsize\textit{By #1}}\par\vbox
to1cm{}}
\def\thanks#1{\global\footnote{#1}}
\def\address#1#2{\global\footnotetext[#1]{\noindent #2}}
\def\maketitle{\renewcommand{\thefootnote}{\arabic{footnote}}\setcounter{footnote}{0}}

\def\runningtitle#1{\markboth{\small\it Algebraic Groups, Algebraic Geometry and
Representation Theory}{\small\it #1}}

\def\Ext{\operatorname{Ext}}         %
\def\Hom{\operatorname{Hom}}         %
\let\hom\Hom                         %

\begin{document}

\title{1}{BGG category $\O$ and $\Z$-graded representation theory}

\runningtitle{Representations of Quivers}

\author{Jun Hu\thanks{Supported by the National Natural
Science Foundation of China (No. 12171029).}}

\address1{Key Laboratory of Mathematical Theory and Computation in Information Security, School of Mathematics and Statistics,  %
Beijing Institute of Technology,                        
Beijing, 100081. P.R. China.}          
\date{}

\maketitle 


\noindent For a finite dimensional complex semisimple Lie algebra $\fg$, we fix a triangular decomposition $\fg=\fn\oplus\fh\oplus\fn^{-}$ and consider its
Bernstein-Gelfand-Gelfand category $\O$ (BGG category $\O$ for short) from \cite{BGG} and its parabolic generalisation by Rocha-Caridi in \cite{R}. These categories play an important role in the study of certain infinite representations of semisimple Lie groups. The characters of simple modules in these categories can be computed using the famous Kazhdan-Lusztig Conjecture (nowadays a theorem). By the celebrated work of Soergel \cite{So}, Beilinson-Ginzburg-Soergel \cite{BGS} and Backelin \cite{Bac}, each block ${\O}_{\lam}$ of $\O$ is equivalent to the category of finite dimensional modules over a finite dimensional Koszul $\C$-algebra $A_\lam$.
The Koszul structure enable people to study the $\Z$-graded representation theory of $A_\lam$. The category $A_\lam\lgmod$ of finite dimensional $\Z$-graded $A_\lambda$-modules is often referred as ${\O}_{\lam}^{\Z}$, the $\Z$-graded version of the BGG category $\O_\lam$, which fits nicely in the framework of Kazhdan-Lusztig theory.

In this chapter we give a brief introduction to the category ${\O}_{\lam}^{\Z}$ as well as its parabolic generalisation and survey some recent development. There are massive references and survey articles on BGG category $\O$, while we will only focus on the $\Z$-graded version ${\O}_{\lam}^{\Z}$ of $\O$, with some emphasis on Soergel's combinatorial $\mathbb{V}$ functor, $\Z$-graded duality as well as $\Z$-graded translation functors. These objects play important role in the $\Z$-graded representation theory of ${\O}_{\lam}^{\Z}$.

The content is organised as follows. In Section one we first recall some preliminary definition and results for the usual BGG category $\O$. Then we give the definition of coinvariant algebra as well as Soergel's combinatorial $\mathbb{V}$ functor. This functor plays a crucial role in the $\Z$-graded representation theory of $A_\lam$. We show that (Corollaries \ref{indec}, \ref{bs}) the image $\mathbb{V}(P(x\cdot\lam))$ of indecomposable projective module $P(x\cdot\lam)$ under the $\mathbb{V}$ functor is an indecomposable gradable $C_\lam$-module. After introducing the parabolic category $\O$, we give the well-know Koszul duality result in Theorem \ref{koszul} and introduce some basic objects in the $\Z$-graded representation theory of $A_\lam$. This section ends with a quick introduction to the theory of projective functors as well as a categorification result for Hecke algebra using indecomposable projective functors. In Section two we first recall the well-known Kazhdan-Lusztig conjecture, which can be seen as a milestone in the development of the theory of BGG category $\O$. Then we introduce some fundamental results about the $\Z$-graded decomposition number for the BGG category $\O$ as well as our recent work on the $\Z$-graded decomposition number and $\Z$-graded inverse decomposition number for the parabolic and general cases. Section three is devoted to the study of the coinvariant algebra $C$ as well as its parabolic invariants $C^J$ under a parabolic subgroup $W_J$ of $W$. We discuss some basic property and study these algebras from a purely algebraic point of view. We show that both $C$ and $C^J$ are graded cellular algebras and equipped with a homogeneous symmetrizing form (Corollaries \ref{tr0}, \ref{cellular2}, Theorem \ref{tr}). We show (Lemma \ref{selfdualC}) that $\mathbb{V}(P(x\cdot\lam))\<-\ell(w_0)+\ell(w_{0,\lam})\>$ is graded self-dual as a graded $C_\lam$-module. As an application, we obtain (Corollary \ref{idem0}, Definition \ref{duality0}) a degree $0$ anti-involution in $A_\lam$ and a $\Z$-graded duality functor in ${\O}_{\lam}^{\Z}$ (which fixes an error in \cite[\S6.1.2]{Str}). The pair of adjoint graded functors given in Proposition \ref{keyprop1} will play a key role in the study of $\Z$-graded translation functors. Section four is devoted to the study of $\Z$-graded lift of translation functors. These $\Z$-graded lifts were constructed in \cite{Str} in the semiregular case and we construct and study them in Lemma \ref{lift1} and Theorem \ref{adj2} in the general case with some new and streamlined argument. We study the action of these graded functors on graded Verma modules, graded simple modules and graded projective modules in Theorem \ref{4results}.

\bigskip
\centerline{Acknowledgements}
\bigskip

The author would like to thank Prof. Wolfgang Soergel, Prof. Ming Fang, Prof. Michael Ehrig, Prof. Wei Xiao and Dr. Huang Lin for their helpful comments and advices.
\bigskip


\section{Preliminary}


Let $\fg$ be a finite dimensional complex semisimple Lie algebra with a triangular decomposition $\fg=\fn\oplus\fh\oplus\fn^{-}$. Let $\Phi$ be the root system of $\fg$ and $W$ the Weyl group of $\fg$. Let $\Pi$ be the set of simple roots in $\Phi$ determined by $\mathfrak{b}:=\fn\oplus\fh$.

\begin{definition} We define the BGG category $\O$ to be the full subcategory of all finitely generated $U(\fg)$-modules $M$ such that $M$ is locally finite for $U(\fn)$ and has weight spaces decomposition under the action of $\fh$.
\end{definition}

For each $\lam\in\fh^*$, we use $L(\lam)$, $\Delta(\lam)$, $P(\lam)$ and $I(\lam)$ to denote the simple module, Verma module, indecomposable projective module, indecomposable injective module labelled by $\lam$ respectively. We refer the readers to \cite{Hum} for their definitions.
For any $w\in W$ and $\lam\in\fh^*$, we define $w\cdot\lam:=w(\lam+\rho)-\rho$, where $\rho$ is the half sum of all the positive roots. We use $\O_\lam$ to denote the Serre subcategory of $\O$ generated by all $L(w\cdot\lam)$ for $w\in W$. We have a decomposition $\O=\oplus_{\lam\in\fh^*/(W,\cdot)}\O_\lam$.

Let $\Pi^\vee:=\{\alpha^\vee|\alpha\in\Pi\}$ the set of simple coroots in $\fh$. A weight $\lam\in\fh^*$ is called integral if $\<\lam,\alpha^\vee\>\in\Z$ for any $\alpha\in\Pi$. By the results in \cite{So}, the study of $\O_\lam$ can be reduced to the case when $\lam$ is integral. Henceforth, we shall always assume that $\lam\in\fh^*$ is integral. In this case, each $\O_\lam$ is a block of $\O$.

\begin{definition} Let $\lam\in\fh^*$ be an integral weight. We call $\lam$ dominant if $\<\lam+\rho,\alpha^\vee\>\geq 0$ for any $\alpha\in\Pi$, and call $\lam$ anti-dominant if $\<\lam+\rho,\alpha^\vee\>\leq 0$ for any $\alpha\in\Pi$.
\end{definition}

An integral weight $\lam$ is said to be regular if $\<\lam+\rho,\alpha^\vee\>\neq 0$ for any $\alpha\in\Phi$, otherwise it is said to be singular. Let $\lam\in\fh^*$ be an integral weight. We define $$
W_\lam:=\{w\in W|w\cdot\lam=\lam\} .
$$
Clearly, $\lam$ is regular if and only if $W_\lam=\{e\}$, and $W_\lam$ is a reflection subgroup of $W$ generated by the reflections it contains by \cite[\S1.12, Theorem]{Hum0}.
If $\lam$ is dominant or anti-dominant, then by \cite[\S1.12, Theorem, Exercise 2]{Hum0}, $W_\lam$ is a standard parabolic subgroup of $W$. In this case, we denote by $w_{0,\lam}$ the unique longest element in $W_\lam$.We use $W^\lam$ (resp., $^{\lam}W$) to denote the set of minimal length left (resp., right) coset representatives of $W_\lam$ in $W$.

Under the dot action of $W$, any $W$-orbit of an integral weight contains a unique dominant weight as well as a unique anti-dominant weight. Thus we can use either the dominant weights or the anti-dominant weights to index the orbits of $W$.

\begin{definition} Let $\lam$ be an integral dominant weight. We define the basic algebra of $\O_\lam$ to be $$
A_\lam:=\End_{\O}\Bigl(\bigoplus_{w\in W^\lam}P(w\cdot\lam)\Bigr) .
$$
\end{definition}
By \cite{Hum}, $A_\lam$ is a finite dimensional quasi-hereditary algebra in the sense of Cline-Parshall-Scott \cite{CPS}. The functor $$
\mathcal{F}_\lam(-):=\Hom_\O\Bigl(\bigoplus_{w\in W^\lam}P(w\cdot\lam),-\Bigr)$$ defines an equivalence between
the category $\O_\lam$ and the category $\rmod A_\lam$ of finite dimensional right $A_\lam$-modules.

Let $K$ be a field. A $\Z$-graded $K$-module a $K$-linear space $M$ which has a direct sum decomposition $M=\oplus_{d\in\Z}M_d$. If $m\in M_d$ for $d\in\Z$, then
we say $m$ is homogeneous of degree $d$ and set $\deg m = d$. If $M$ is a graded $K$-module, let $\underline{M}$ be the ungraded $K$-module obtained by forgetting the grading on $M$. If $M$ is a graded $K$-module and $k\in\Z$, let $M\<k\>$ be the graded $K$-module obtained by shifting the grading on $M$ up by $k$; that is,
$M\<k\>_d:=M_{d-k}$, for $d\in\Z$. A graded $K$-algebra is a unital associative $K$-algebra $A=\oplus_{d\in\Z}A_d$ which is a graded $K$-module
such that $A_dA_e\subseteq A_{d+e}$, for all $d,e\in\Z$. A graded (right) $A$-module is a graded $K$-module $M$ such that $M$ is an $A$-module
and $M_dA_e\subseteq M_{d+e}$, for all $d,e\in Z$. Graded submodules, graded left $A$-modules and so on are all defined in the obvious way. Let $A$ be a graded $K$-algebra. We use $\rgmod A$ to denote the category of all finite dimensional graded right $A$-modules together with degree preserving homomorphisms; that is, $$
\ho_A(M,N) = \{f\in \Hom_A(\underline{M},\underline{N}) | \text{$f(M_d)\subseteq N_d$ for all $d\in\Z$}\},
$$
for all $M, N\in \rgmod A$. The elements of $\ho_A(M,N)$ are homogeneous maps of degree $0$. More generally, if $f\in\ho_A(M\<d\>,N)\cong
\ho_A(M,N\<-d\>)$ then $f$ is a homogeneous map from $M$ to $N$ of degree $d$ and we write $\deg f = d$. In a similar way, one can define the category $A\lgmod$
of  finite dimensional graded left $A$-modules.

Let $R:=S(\fh^*)$ be the symmetric algebra on $\fh^*$. There is a natural action of $W$ on $\fh$ and hence on $\fh^*$ and $S(\fh^*)$.
We regard $S(\fh^*)$ as a $\Z$-graded algebra by endowing each element in $\fh^*$ degree $2$. By definition, $S(\fh^*)=\oplus_{j\geq 0}S_j(\fh^*)$.

\begin{definition}\label{iw} Let $I_W$ denote the homogeneous ideal of $S(\fh^*)$ generated by $R_+^W:=\oplus_{j\geq 1}{S_j}(\fh^*)^W$. We define
$$
C:={S(\fh^*)}/{I_W},
$$
and call $C$ the coinvariant algebra of $W$.
\end{definition}

Let $S:=\{s_\alpha|\alpha\in\Pi\}$ be the set of simple reflections in $W$. Let $\ell(-)$ be the length function on $W$ with respect to $S$.
Let $w_0$ be the unique longest element in $W$. Let $\lam$ be an integral dominant weight as before. Then $P(w_0\cdot\lam)$ is the unique indecomposable projective and injective module in $\O_\lam$.

\begin{lemma}\text{(\cite{So})}\label{c0} Let $\lam$ be an integral dominant weight. There is a $\C$-algebra isomorphism $\theta: C_\lam:=C^{W_\lam}\cong\End_{\O}\bigl(P(w_0\cdot\lam)\bigr)$.
\end{lemma}

\begin{dfnthm}\text{(\cite{So})}\label{ff} Let $\lam$ be an integral dominant weight. We call the following functor $$\begin{aligned}
\mathbb{V}_\lam:\,\, & \O_\lam\rightarrow\rmod C_\lam \\
& M\mapsto\Hom_{\O}\bigl(P(w_0\cdot\lam), M\bigr), \quad\forall\, M\in\O_\lam ,
\end{aligned}
$$
the combinatorial $\mathbb{V}$-functor. Then $\mathbb{V}_\lam$ is exact and fully faithful on projectives. That is, for any $x,y\in W^\lam$, $$
\Hom_{\O}\bigl(P(x\cdot\lam),P(y\cdot\lam)\bigr)\cong\Hom_{C^{\lam}}\bigl(\mathbb{V}_\lam(P(x\cdot\lam)),\mathbb{V}_\lam(P(y\cdot\lam))\bigr) .
$$
\end{dfnthm}

\begin{remark} 1) The module $\mathcal{F}_\lam(P(w_0\cdot\lam))$ is actually a faithful projective and injective module over $A_\lam$. Definition and Theorem \ref{ff} is a consequence of the double centralizer property with respect to this projective and injective module. K\"onig, Xi and Slung{\aa}rd deduced this double centralizer property directly from the fact that $A_\lam$ has dominant dimension at least two (\cite{KSX}). Such kind of double centralizer property actually holds for any faithful tilting modules, see \cite{HuX} for more details.

2) The combinatorial $\mathbb{V}$-functor $\mathbb{V}_\lam$ is an analogue of the well-known Schur functor in algebraic Lie theory. One can think of each $\mathbb{V}_\lam(P(x\cdot\lam))$ as an analogue of the Young module for the symmetric groups or Hecke algebras of type $A$, see \cite{Gr}.
\end{remark}

\begin{corollary}\label{indec} Let $\lam$ be an integral dominant weight. Then for each $x\in W^\lam$, $\mathbb{V}_\lam(P(x\cdot\lam))$ is an indecomposable $C_\lam$-module.
Moreover, for any $x,y\in W^\lam$, $\mathbb{V}_\lam(P(x\cdot\lam))\cong \mathbb{V}_\lam(P(y\cdot\lam))$ if and only if $x=y$.
\end{corollary}

\begin{proof} This is clear by Definition and Theorem \ref{ff}, $$
\End_{C_\lam}\bigl(\mathbb{V}_\lam(P(x\cdot\lam))\bigr)\cong\End_\O\bigl(P(x\cdot\lam)\bigr)\cong\End_{A_\lam}\bigl(\mathcal{F}_\lam(P(x\cdot\lam))\bigr),
$$
while the latter is a local ring (because $\mathcal{F}_\lam(P(x\cdot\lam))$ is an indecomposable $A_\lam$-module).
\end{proof}

\begin{theorem}\text{(\cite{So})}\label{BS0} Let $\lam$ be an integral dominant weight and $x\in W^\lam$. Let $s_r\cdots s_2s_1$ be a reduced expression of $x$. Then the module $\mathbb{V}_\lam(P(x\cdot\lam))$ is isomorphic to a direct summand of \begin{equation}\label{BS}
C_\lam\otimes_{(C_\lam)^{s_1}}C_\lam\otimes_{(C_\lam)^{s_2}}C_\lam\otimes_{(C_\lam)^{s_3}}\cdots C_\lam\otimes_{(C_\lam)^{s_r}}\C .
\end{equation}
If furthermore $\lam$ is regular, then it occurs with multiplicity one and all the other direct summands are of the form $\mathbb{V}_\lam(P(y\cdot\lam))$ with $x>y\in W^\lam$.
\end{theorem}

In the study of $\Z$-graded representation theory of a graded algebra $A$, it is important to understand whether a module or a functor in $\underline{A}\lmod$ is gradable (i.e., allows a $\Z$-graded lift). We refer the readers to \cite{GG} and \cite[\S3.1,3.2]{Str} for precise definitions and more details.

\begin{corollary}\label{bs} Let $\lam$ be an integral dominant weight and $x\in W^\lam$. Then $\mathbb{V}_\lam(P(x\cdot\lam))$ can be endowed with a $\Z$-grading such that it becomes a $\Z$-graded $C_\lam$-module. In other words, $\mathbb{V}_\lam(P(x\cdot\lam))$ is $\Z$-gradable.
\end{corollary}

\begin{proof} This follows from Theorem \ref{BS0}, the fact that (\ref{BS}) is a $\Z$-graded $C_\lam$-module and a standard result \cite[Theorem 3.3]{GG}.
\end{proof}

Let $I$ be a subset of $\Pi$. Let $\frp_I$ be the standard parabolic subalgebra of $\frg$ corresponding to $I$ with Levi decomposition $\frp_I:=\frl_I\oplus\fru_I$, where $\frl_I$ is the standard Levi subalgebra and $\fru_I$ the nilpotent radical of $\frp_I$. The parabolic category $\O^{\frp_I}$ associated to $\frp_I$ is defined to be the category of all finitely generated $\frg$-modules $M$, which are semisimple as $\frl_I$-modules and locally $U(\fru_I)$-finite. If $I=\emptyset$, $\O^{\frp_I}$ is the usual BGG category $\O$; if $I=\Pi$, then $\O^{\frp_I}$ is the category of finite dimensional $\fg$-modules. In general, $\O^{\frp_I}$ is a full subcategory of $\O$.
If $\mu\in\Lam$ is dominant, then we use $\frp(\mu)=\frp_{I_\mu}$ to denote the standard parabolic subalgebra of $\frg$ corresponding to the subset $$I_\mu:=\{\alpha\in\Pi\,|\,\<\mu+\rho,\alpha^\vee\>=0\}$$ of $\Pi$. We define
$\O^\mu:=\O^{\frp(\mu)}$ and $\O_\lam^\mu:=\O^\mu\cap\O_{\lam}$ for any $\lam\in\Lam$.

Let $$
\Lam^+:=\{\nu\in\Lam|\<\nu,\alpha^\vee\>\geq 0,\forall\,\alpha\in\Pi\},\quad
\Lam_I^+:=\{\nu\in\Lam|\<\nu,\alpha^\vee\>\geq 0,\forall\,\alpha\in I\} .
$$
Then for any $w\in W$, $L(w\cdot\lam)\in\O_\lam^\mu$ if and only if $w\cdot\lam\in\Lam_{I_\mu}^+$.

We fix two integral weights $\lam,\mu$. Define \begin{equation}\label{ij1}
I=\bigl\{\alpha\in\Pi\bigm|\langle\mu+\rho,\alpha^\vee\rangle=0\bigr\},\quad J:=\bigl\{\alpha\in\Pi\bigm|\langle\lambda+\rho,\alpha^\vee\rangle=0\bigr\}.
\end{equation}
We define \begin{equation}\label{ij}
{}^IW^J=\{w\in{}^IW\mid \ell(w)+1=\ell(ws_\alpha)\ \mbox{and}\ ws_\alpha\in {}^IW,\ \mbox{for all}\ \alpha\in J\}.
\end{equation}
Note that ${}^IW^J$ is a (possibly empty) subset of ${}^I W\cap W^J$ (the set of minimal length $(W_I,W_J)$ double cosets representatives in $W$). Suppose that both $\lam,\mu$ are  integral dominant weights. Then every integral weight $\mu\in W\cdot\lambda\cap\Lambda_I^+$ can be uniquely written in the form $\mu=w\cdot\lambda$ for some $w\in{}^IW^J$. Therefore, the simple modules in $\O_\lam^\mu$ are given by $\{L(w\cdot\lam)|w\in {}^IW^J\}$. For each $w\in {}^IW^J$, let $P^\mu(w\cdot\lam)$ be the projective cover of $L(w\cdot\lam)$ in $\O_\lam^\mu$, let $\Delta^\mu(w\cdot\lam)\in\O_\lam^\mu$ be the parabolic Verma module which is the maximal quotient of the usual Verma module $\Delta(w\cdot\lam)$ which lies in $\O_\lam^\mu$. It
is a quotient of $P^\mu(w\cdot\lam)$ and has a unique simple head $L(w\cdot\lam)$.

To ease the notations, we set $$\begin{aligned}
& A:=A_0,\quad \mathbb{V}:=\mathbb{V}_0,\quad P:=\oplus_{x\in W}P(x\cdot 0),\quad P_\lam:=\oplus_{x\in W^\lam}P(x\cdot\lam) ,\\
& P^\mu_\lam:=\oplus_{x\in {}^IW^J}P^\mu(x\cdot\lam),\quad A_\lam^\mu:=\End_\O(P^\mu_\lam) .
\end{aligned}
$$

\begin{theorem}\text{(\cite{Bac}, \cite{BGS})}\label{koszul} Let $\lam,\mu$ be two integral dominant weights. Let $I,J$ be defined as in (\ref{ij1}). Then $A_{\lam}^\mu$ is a finite dimensional Koszul algebra and there is a $\Z$-graded $\C$-algebra isomorphism: $$
A_{\lam}^\mu\cong\Ext_{\O_{-w_0\mu}^{\lam}}^{\bullet}\Bigl(\oplus_{x\in {}^JW^{-w_0I}}L(x\cdot(-w_0\mu)),\oplus_{x\in {}^JW^{-w_0I}}L(x\cdot(-w_0\mu))\Bigr),
$$
and there is a Koszul duality functor $K$ which gives an equivalence from $D^b(A_\lam^\mu\lgmod)$ onto $D^b(A_{-w_0\mu}^{\lam}\lgmod)$, and it sends simple (resp., indecomposable projective) $A_\lam^\mu$-modules to indecomposable injective (resp., simple) $A_{-w_0\mu}^{\lam}$-modules.
 \end{theorem}
The algebra $A_{-w_0\mu}^{\lam}$ (which is the basic algebra of $\O_{-w_0\mu}^{\lam}$) is called the Koszul dual of $A_{\lam}^\mu$ and denoted by $(A_\lam^\mu)^{!}$.

The Koszul structure on $A_{\lam}^\mu$ has many strong implication on the structure of the category $\O_\lam^\mu$. In particular, one can study the $\Z$-graded representation theory of $\O_\lam^\mu$.
For each $w\in {}^IW^J$, we use $\mathbb{L}(w\cdot\lam)$ to denote the one-dimensional simple module in $\rgmod A_\lam^\mu$ which is concentrated in degree $0$ and hence a graded lift of $L(w\cdot\lam)$. Let $\mathbb{P}^\mu(w\cdot\lam)$ be the projective cover of $\mathbb{L}(w\cdot\lam)$ in $\rgmod A_\lam^\mu$ which has a unique graded simple quotient $\mathbb{L}(w\cdot\lam)$. This is a graded lift of $P^\mu(w\cdot\lam)$.

\begin{lemma}\text{(\cite{BGS},\cite{M3})} For each $w\in {}^IW^J$, the parabolic Verma module $\Delta^\mu(w\cdot\lam)$ allows a $\Z$-graded lift $\mathbb{\Delta}^\mu(w\cdot\lam)$ such that it is a graded quotient of $\mathbb{P}^\mu(w\cdot\lam)$. In particular, it has a unique graded simple quotient which is isomorphic to $\mathbb{L}(w\cdot\lam)$.
\end{lemma}

We refer the readers to \cite{CM}, \cite{M2}, \cite{M4}, \cite{MOS} (and the references therein) for important application of Koszul duality functors in the study of BGG category $\O$.

Let $\lam,\mu$ be two integral dominant weights. Recall that (\cite{BG}) a functor $\theta: \O_\lam\rightarrow\O_\mu$ is called projective if it is isomorphic to a direct summand of $V\otimes_{\C}-$ for some finite dimensional $\fg$-module $V$. By one of the main result in \cite{BG}, for any $W_\lam$-anti-dominant $\nu\in W\cdot\mu$ (i.e., $\nu\leq W_\lam\cdot\nu$), there is a unique indecomposable projective functor $\theta_{\lam,\nu}$ such that $\theta_{\lam,\nu}(\Delta(\lam))=P(\nu)$. Furthermore, every indecomposable projective functor from $\O_\lam$ to $\O_\mu$ is isomorphic to $\theta_{\lam,\nu}$ for some $W_\lam$-anti-dominant $\nu\in W\cdot\mu$.

Suppose that $\lam$ is a regular integral dominant weight. Then the indecomposable projective endofunctors of $\O_\lam$ are in bijection with the elements in $W$. For each $w\in W$, we denote by $\theta_w$ the unique indecomposable projective endofunctor of $\O_\lam$ such that $\theta_w(\Delta(\lam))=P(w\cdot\lam)$.

Let $w\in W$ and $s_1s_2\cdots s_r$ be a reduced expression of $w$, where $s_i\in S$ for each $i$. By \cite[Proposition 5.5]{M4}, $\theta_w$ is a direct summand of $\theta_{s_r}\cdots\theta_{s_2}\theta_{s_1}$, and all the other direct summands of are isomorphic to $\theta_x$ for $x<w$. By \cite{Str}, each $\theta_{s_i}$ is gradable. It follows that $\theta_w$ is gradable too.

\begin{definition} Let $\lam$ be a regular integral dominant weight. For each $w\in W$, we define $\bbtheta_w: \rgmod A_\lam\rightarrow\rgmod A_\lam$ to be the unique $\Z$-graded lift of $\theta_w$ such that
$\bbtheta_w(\mD(\lam))=\mP(w\cdot\lam)$.
\end{definition}

Let $v$ be an indeterminate over $\Z$ and $q:=v^2$. Let $\HH_v(W)$ be the Iwahori-Hecke algebra associated to $(W,S)$ with Hecke parameter $q$. By definition, $\HH_v(W)$ is a free $\Z[v,v^{-1}]$-module with standard basis $\{T_w|w\in W\}$ and multiplication given by $$
T_xT_y=T_{xy},\,\,\text{if $\ell(xy)=\ell(x)+\ell(y)$,}\quad T_s^2=(v-v^{-1})T_s+1,\,\,\forall\,s\in S .
$$
The algebra $\HH_v(W)$ is a deformation of the group algebra $\Z[W]$. There is a unique $\Z$-linear involution ``$-$'' on $\HH_v(W)$ which maps $v^k$ to $v^{-k}$ for all $k\in\Z$ and $T_w$ to $T_{w^{-1}}^{-1}$ for all $w\in W$. By a well-known result of Kazhdan and Lusztig \cite{KL}, $\HH_v(W)$ has a unique $\Z[v,v^{-1}]$-basis $\{C_w|w\in W\}$ such that \begin{enumerate}
\item[1)] for each $w\in W$, $\overline{C_w}=C_w$, and
\item[2)] we have $$
C_w=T_w+\sum_{w>y\in W}v^{\ell(y)-\ell(w)}P_{y,w}(v^2)T_y ,
$$
where $P_{y,w}$ is a polynomial in $q$ of degree $\leq (\ell(w)-\ell(y)-1)/2$, and $P_{w,w}:=1$.
\end{enumerate}
In particular, $C_w\in T_w+\sum_{w>y\in W}v^{-1}\Z[v^{-1}]T_y$. The polynomials $\{P_{y,w}|y\leq w\in W\}$ are the famous Kazhdan-Lusztig polynomials, while the element $C_w$ is the famous Kazhdan-Lusztig basis element of $\HH_v(W)$ (which was denoted by $C'_w$ in \cite{KL}).

Let $\lam$ be a regular integral dominant weight. For each $w\in W$ and $k\in\Z$, we define $v^k\bbtheta_w:=\bbtheta_w\<-k\>$. Let $\mathcal{P}$ be the additive category of graded projective endofunctors of $\O_\lam$ and $[\mathcal{P}]$ be its enriched split Grothendieck group, where the $\Z[v,v^{-1}]$-module structure on $[\mathcal{P}]$ is defined by $v^k[\bbtheta_w]:=[\bbtheta_w\<-k\>]$ for $k\in\Z$ and $w\in W$.

\begin{theorem}\text{(\cite[Theorem 7.11]{M4})} With the notations as above, the map which sends $\bbtheta_w$ to $C_w$ for each $w\in W$ can be extended uniquely to an anti-isomorphism of $\Z[v,v^{-1}]$-algebras between
$[\mathcal{P}]$ and the Hecke algebra $\HH_v(W)$.
\end{theorem}

In other words, the above theorem says that the additive category of the graded projective endofunctors of $\O_\lam$ gives a categorification of the Iwahori-Hecke algebra $\HH_v(W)$. For more application of Kazhdan-Lusztig theory in the study of $\Z$-graded version of BGG category $\O$, see \cite{CM4}, \cite{M1} and \cite{M4}.

\section{$\Z$-graded decomposition numbers}

One of the central problem in the study of BGG category $\O$ is to determine the character of each simple module, or equivalently, to compute the composition multiplicity of each simple module in any given Verma module. This was formulated (\cite{KL}) as the following famous Kazhdan-Lusztig conjecture, which was proved by Beilinson and Bernstein \cite{BB0}, and by Brylinski and Kashiwara \cite{BK}.

\begin{theorem}[Kazhdan-Lusztig Conjecture] Let $\lam$ be an integral anti-dominant weight and $x\in W^\lam$. Then in the Grothendieck group of $\O_\lam$ we have $$\begin{aligned}
\left[ \Delta(x\cdot\lam) \right] & = \sum_{x>y\in W^\lam}P_{xw_0,yw_0}(1)\left[{L}(y\cdot\lam)\right] ,\\
\left[L(x\cdot\lambda)\right] &=\sum_{y\in W^J}\sum_{z\in W_J}(-1)^{\ell(x)+\ell(z)-\ell(y)}P_{yz,x}(1)\left[\Delta(y\cdot\lambda)\right].\end{aligned} $$
\end{theorem}

The above theorem is a shadow of a $\Z$-grading structure on $\O_\lam$ which is controlled by Kazhdan-Lusztig theory. It is better to understand the Kazhdan-Lusztig Conjecture in the framework of $\Z$-graded representation theory of $A_\lam$, which was worked out in \cite{BGS}.

Let $\lam,\mu$ be two integral anti-dominant weights and $x,y\in W^\lam$. Let $I,J$ be defined as in (\ref{ij1}). Recall that the $\Z$-graded projective module $\mP^\mu(x\cdot\lam)$ has a $\Z$-graded $\mD$-flag, i.e., a $\Z$-graded $A_\lam^\mu$-submodule filtration $0=P_0\subset P_1\subset\cdots\subset P_m=\mP^\mu(x\cdot\lam)$ such that each successive quotient is isomorphic to some $\mD^\mu(z\cdot\lam)\<k\>$ for some $z\in {}^IW^J$ and $k\in\Z$.

\begin{definition} We use $\bigl(\mP^\mu(x\cdot\lam):\mD^\mu(z\cdot\lam)\<k\>\bigr)$ to denote the multiplicity of $\mD^\mu(z\cdot\lam)\<k\>$ which occurs as a successive quotient in any given $\Z$-graded $\mD$-flag of $\mP^\mu(x\cdot\lam)$. Define $$
\bigl(\mP^\mu(x\cdot\lam):\mD^\mu(y\cdot\lam)\bigr)_v:=\sum_{k\in\Z}\bigl(\mP^\mu(x\cdot\lam):\mD^\mu(y\cdot\lam)\<k\>\bigr)v^k\in\N[v,v^{-1}],
$$
and the $\Z$-graded decomposition numbers $$\begin{aligned}
d^\mu_{x\cdot\lam,y\cdot\lam}(v)&:=\bigl[\mD^\mu(x\cdot\lam):\mL(y\cdot\lam)\bigr]_v\\
&=\sum_{k\in\Z}\bigl[\mD^\mu(x\cdot\lam):\mL(y\cdot\lam)\<k\>\bigr]v^k\in\N[v,v^{-1}] ,
\end{aligned}$$
where $\bigl[\mD^\mu(x\cdot\lam):\mL(y\cdot\lam)\<k\>\bigr]$ denotes the graded composition multiplicity of $\mL(y\cdot\lam)\<k\>$ in $\mD^\mu(x\cdot\lam)$.
\end{definition}

\begin{lemma}[Graded BGG reciprocity] With the notations as above, we have that $\bigl(\mP^\mu(x\cdot\lam):\mD^\mu(y\cdot\lam)\bigr)_v=\bigl[\mD^\mu(y\cdot\lam):\mL(x\cdot\lam)\bigr]_v$.
\end{lemma}

Let $w_{I}$ be the unique longest element in $W_I:=W_\mu$.
The following theorem gives the $\Z$-graded decomposition numbers of ${}^\Z\O_\lam^\mu$ in the case when either $I=\emptyset$ or $J=\emptyset$.

\begin{theorem}\text{(\cite[Theorem 3.11.4(i),(ii)]{BGS})}\label{2cases} Let $\lam,\mu$ be two integral anti-dominant weights and $x,y\in {}^IW^J$, where $I,J$ are defined as in (\ref{ij1}). \begin{enumerate}
\item[1)] if $I=\emptyset$, then in the enriched Grothendieck group $K_0(\rgmod A_{\lambda})$ we have $$
[\mathbb{\Delta}(x\cdot\lambda)]=\sum_{y\in W^J}v^{\ell(x)-\ell(y)}P_{xw_0,yw_0}(v^{-2})[\mathbb{L}(y\cdot\lambda)] .
$$
\item[2)] if $J=\emptyset$, then in the enriched Grothendieck group $K_0(\rgmod A_{0}^\mu)$ we have $$
[\mathbb{\Delta}^\mu(w_Ixw_0\cdot 0)]=\sum_{y\in ^{I}W}\sum_{z\in W_I}(-1)^{\ell(z)}v^{\ell(x)-\ell(y)}P_{zw_Ixw_0,w_Iyw_0}(v^{-2})[\mathbb{L}(w_Iyw_0\cdot 0)] .
$$
\end{enumerate}
\end{theorem}

\begin{remark} To identify the formula 2) in the above theorem with \cite[Theorem 3.11.4(i),(ii)]{BGS}, one should note that \cite{BGS} used an integral dominant weight $\psi$ to index a parabolic subalgebra $\mathfrak{q}(\psi)$, while our convention in the above theorem is to use an integral anti-dominant weight $\mu$ to index a parabolic subalgebra $\frp_I$. Thus one should replace the elements $x^{-1}w_0, y^{-1}w_0, z$ in the notation of  \cite[Theorem 3.11.4(i),(ii)]{BGS} with the elements $w_0(w_Ixw_0)w_0, w_0(w_Iyw_0)w_0, w_0zw_0$ in the above theorem.
\end{remark}

\begin{corollary}\label{gradedlength1} Let $\lam,\mu$ be two integral anti-dominant weights and $x,y\in {}^IW^J$, where $I,J$ are defined as in (\ref{ij1}). Suppose that $I=\emptyset$. Then the graded length of $\mD(x\cdot\lam)$ is equal to $\ell(x)$.
\end{corollary}

\begin{proof} This follows from Theorem \ref{2cases} by calculating the maximal power of $v$ appeared in the coefficient.
\end{proof}

Recently, Theorem \ref{2cases} was generalized by Xiao Wei and the author to the general case \cite{HuXiao}, where we also obtain explicit formulae for the inverse graded decomposition numbers in general cases.

\begin{theorem}\text{(\cite{HuXiao})} Let $\lam,\mu$ be two integral anti-dominant weights and $x,y\in {}^IW^J$, where $I,J$ are defined as in (\ref{ij1}). Then in the enriched Grothendieck group $K_0(\rgmod A_{\lambda}^\mu)$ we have $$\begin{aligned}
\bigl[\mathbb{\Delta}^\mu(w_Ix\cdot\lambda)\bigr] & =\sum_{y\in ^{I}W^J}\sum_{z\in W_I}(-1)^{\ell(z)}v^{\ell(x)-\ell(y)}P_{zw_Ixw_0,w_Iyw_0}(v^{-2})\bigl[\mathbb{L}(w_Iy\cdot\lambda)\bigr] ,\\
\bigl[{\mathbb{L}}(w_Iy\cdot\lambda)\bigr] & =\sum_{x\in ^{I}W^J}\sum_{z\in W_J}(-1)^{\ell(y)+\ell(z)-\ell(x)}v^{\ell(y)-\ell(x)}P_{w_Ixz,w_Iy}(v^{-2})
\bigl[{\mathbb{\Delta}}^\mu(w_Ix\cdot\lambda)\bigr] .
\end{aligned}$$
\end{theorem}

\begin{corollary}[Positivity] Let $\lam,\mu$ be two integral anti-dominant weights and $x,y\in {}^IW^J$. Then for any $k\in\Z$, $\bigl[\mathbb{\Delta}^\mu(w_Ix\cdot\lambda):\mathbb{L}(w_Iy\cdot\lambda)\<k\>\bigr]\neq 0$ only if
$k\geq 0$ and $x\geq y$. Moreover, $\bigl[\mathbb{\Delta}^\mu(w_Ix\cdot\lambda):\mathbb{L}(w_Iy\cdot\lambda)\bigr]\neq 0$ only if $x=y$, and in this case $\bigl[\mathbb{\Delta}^\mu(w_Ix\cdot\lambda):\mathbb{L}(w_Ix\cdot\lambda)\bigr]=1$.
\end{corollary}

\begin{definition} Let $\lam,\mu$ be two integral anti-dominant weights and $x,y\in {}^IW^J$. We define the graded Cartan number $$
c^\mu_{x\cdot\lam,y\cdot\lam}(v):=\sum_{k\in\Z}\bigl[\mP^\mu(x\cdot\lam):\mL(y\cdot\lam)\<k\>\bigr]v^k\in\N[v,v^{-1}] .
$$
\end{definition}

By the graded BGG reciprocity, we have that $$
c^\mu_{x\cdot\lam,y\cdot\lam}(v)=\sum_{z\in {}^IW^J}d^\mu_{z\cdot\lam,x\cdot\lam}(v)d^\mu_{z\cdot\lam,y\cdot\lam}(v)\in\N[v,v^{-1}].
$$
In particular, the graded Cartan matrix $(c^\mu_{x\cdot\lam,y\cdot\lam}(v))_{x,y}$ is symmetric.

\begin{corollary}\label{gradedlength2} Let $\lam,\mu$ be two integral anti-dominant weights and $x,y\in {}^IW^J$, where $I,J$ are defined as in (\ref{ij1}). Suppose that $I=\emptyset$. Then the graded length of $\mP(x\cdot\lam)$ is equal to $2(\ell(w_0)-\ell(w_{0,\lam}))-\ell(x)$.
\end{corollary}

\begin{proof} This follows from the above formula for $c^\mu_{x\cdot\lam,y\cdot\lam}(v)$ and Theorem \ref{2cases}.
\end{proof}

\begin{corollary}[Positivity]\label{pos2} Let $\lam,\mu$ be two integral anti-dominant weights and $x,y\in {}^IW^J$. Then for any $k\in\Z$, $\bigl[\mathbb{P}^\mu(w_Ix\cdot\lambda):\mathbb{L}(w_Iy\cdot\lambda)\<k\>\bigr]\neq 0$ only if
$k\geq 0$. Moreover, $\bigl[\mathbb{P}^\mu(w_Ix\cdot\lambda):\mathbb{L}(w_Iy\cdot\lambda)\bigr]\neq 0$ only if $x=y$, and in this case $\bigl[\mathbb{P}^\mu(w_Ix\cdot\lambda):\mathbb{L}(w_Ix\cdot\lambda)\bigr]=1$.
\end{corollary}

\begin{corollary}\label{p2p} Let $\lam,\mu$ be two integral anti-dominant weights and $x,y\in {}^IW^J$. Let $f\in\hom_{A_\lam^\mu}(\mP^\mu(x\cdot\lam),\mP^\mu(y\cdot\lam))$ be a homogeneous homomorphism. Then $\deg f\geq 0$. Moreover $\deg f=0$ if and only if $x=y$ and $f\in\mathbb{C}\id_{\mP^\mu(x\cdot\lam)}$.
\end{corollary}

\begin{proof} This is a direct consequence of Corollary \ref{pos2}.
\end{proof}

\begin{lemma}\text{(\cite[Erweiterungssatz 5]{So})} The isomorphism in Lemma \ref{c0} can be upgraded into a $\Z$-graded algebra isomorphism between $C_\lam$ and $\End_{\O}(P(w_0\cdot\lam))$.
\end{lemma}

Let $\lam$ be an integral dominant weight and $x\in W^\lam$. We have now two $\Z$-grading structure on $\mathbb{V}(P(x\cdot\lam)$: one comes from Corollary \ref{bs}, another comes from the Koszul grading on $\rgmod A_\lam$.

\begin{corollary} Let $\lam$ be an integral dominant weight and $x\in W^\lam$. Then the above two gradings on $\mathbb{V}_\lam(P(x\cdot\lam))$ coincides with each other up to a grading shift.
\end{corollary}

\begin{proof} By Corollary \ref{indec}, $\mathbb{V}_\lam(P(x\cdot\lam))$ is an indecomposable $C_\lam$-module. Therefore, the above two gradings on $\mathbb{V}_\lam(P(x\cdot\lam))$ coincides with each other up to a grading shift.
\end{proof}


\section{Coinvariant algebras and $\Z$-graded duality}

In this section we shall study the structure of the coinvariant algebra $C$ as well as its parabolic invariants $C^J$. We shall show that $C^J$ is a graded cellular algebra and equipped with a homogeneous symmetrizing form. These results are used to define a $\Z$-graded duality on $\rgmod A_\lam$.

Let $\Phi^+$ be the set of positive roots in $\Phi$.  For each $\beta\in\Phi^+$, let $s_\beta$ be the corresponding reflection in $W$. Let $G$ be the connected and simply connected complex semisimple Lie group with Lie algebra $\fg$. Let $B$ be the Borel subgroup of $G$ such that its Lie algebra is $\mathfrak{b}$. Let $G/B$ be the associated flag manifold, which plays a central role in many aspects of the representation theory of $G$ and $\fg$. Let $\H^*(G/B)$ be the cohomology of $G/B$ with coefficients in $\C$, which is a $\C$-algebra with multiplication given by the cup product. It is well-known (\cite[\S1]{BGG0}) that $\H^k(G/B)\neq 0$ only if $k\in 2\N$.
The $\C$-algebra $\H^*(G/B)$ is naturally $\Z$-graded if we endow any nonzero elements in $\H^k(G/B)$ with degree $k$ for $k\geq 0$. For each $\alpha\in\Pi$, let $\omega_\alpha\in\fh^*$ be the corresponding fundamental dominant weight which is defined by $\<\omega_\alpha,\beta^\vee\>=\delta_{\alpha,\beta}$ for any $\beta\in\Pi$.

\begin{theorem}\text{(\cite[Theorem 3.17]{BGG0},\cite[Chapter IV, (1.10),(2.9),(3.6)]{Hi})}\label{borel} 1) For each $w\in W$, there is a homogeneous element $X_w$ of degree $2\ell(w)$ in $C$, such that for any $\alpha\in\Pi$, $$
X_{s_\alpha}X_w=\sum_{\substack{\beta\in\Phi^+\\ \ell(ws_\beta)=\ell(w)+1}}\<\omega_\alpha,\beta^\vee\>X_{ws_\beta},
$$
Moreover, the elements $\{X_w|w\in W\}$ form a $\C$-basis of $C$;

2) Let $w,w'\in W$ with $\ell(w)=\ell(w')$. Then\footnote{We remark that the condition $\ell(w)=\ell(w')$ was mistakenly ignored in \cite[III, (2.9)]{Hi}, see \cite[Theorem 3.17(ii)]{BGG0}.} $X_w X_{w_0w'}=\delta_{w,w'}X_{w_0}$;

3) There is a $\Z$-graded $\C$-algebra isomorphism $c: C\cong \H^*(G/B)$. In particular, $\H^0(G/B)=\C$,  $\H^{2\ell(w_0)}(G/B)=\C c(X_{w_0})$, $\H^{2}(G/B)$ has a basis $\{c(X_{s_\alpha})|\alpha\in\Pi\}$, and $\H^*(G/B)$ is generated by $\H^0(G/B)$ and $\H^2(G/B)$;

4) For any homogeneous elements $x,y\in\H^*(G/B)$, $xy=0$ whenever $\deg x+\deg y>2\ell(w_0)$.
\end{theorem}

\begin{remark} 1) One should identify the element $X_w$ with the element $P_{w}$ in \cite[Theorem 3.14]{BGG0}. The image $c(P_w)\in\H^{2\ell(w)}(G/B)$ of $P_w\in C$ is called the Schubert cohomology class corresponding to $w$, which is dual to the Schubert homology class corresponding to $w$. For any $w,w'\in W$, the product $X_wX_{w'}$ can be expressed as an non-negative integral linear combination of some basis elements $\{X_{w''}|w''\in W\}$. The description of those integral coefficients is usually called Schubert calculus.

2) The third isomorphism in the above theorem is usually referred as Borel's picture for the cohomology of $G/B$.
\end{remark}

Henceforth, we use the map $c$ to identify $\H^*(G/B)$ with the coinvariant algebra $C$.

Recall that $R=S(\mathfrak{h}^*)$. For each $\alpha\in\Pi$, there is a Demazure operator (\cite[Definition 3.2]{BGG}, \cite[Chapter IV, \S1]{Hi}) $\Delta_{s_\alpha}: R\rightarrow R$ of degree $-2$, defined by $$
\Delta_{s_\alpha}(f):=\frac{f-s_\alpha f}{\alpha}\in R .
$$
If $w=s_{j_1}\cdots s_{j_k}$ is a reduced expression of $w\in W$, then by \cite[Theorem 3.4]{BGG} and \cite[Proposition 1.2]{Hi}), $\Delta_w:=\Delta_{s_{j_1}}\cdots\Delta_{s_{j_k}}$ depends only on $w$ but not on the choice of the reduced expression of $w$. If $f\in R^{W}$, then $\Delta_w(fg)=f\Delta_w(g)$ for any $g\in R$. Thus $\Delta_w$ naturally induces an endomorphism of $C:=R/I_W$ (Definition \ref{iw}), which will be still denoted by $\Delta_w$.

\begin{lemma}\text{\cite[Chapter IV, Proposition 2.8]{Hi}}\label{Dema} 1) Let $w,u\in W$. Then $\Delta_w\Delta_u=\Delta_{wu}$ if $\ell(wu)=\ell(w)+\ell(u)$; and $\Delta_w\Delta_u=0$ otherwise;

2) Let $w\in W$ and $D:=\frac{1}{|W|}\prod\limits_{\alpha\in \Phi^+}\alpha\in R$. Then $X_w=\Delta_{w^{-1}w_0}\overline{D}$, where $\overline{D}$ denotes the natural image of $D$ in $C$.
\end{lemma}

\begin{corollary}\label{tr0} 1) The set $\{X_w|w\in W\}$ is a $\Z$-graded cellular basis of $C$ which makes $C$ into a $\Z$-graded cellular algebra in the sense of \cite{HuMathas:GradedCellular}. Moreover, each cell module of $\H^*(G/B)$ is one-dimensional.

2) The map $X_w\mapsto\delta_{w,w_0}$, $\forall\,w\in W$, can be extends uniquely to a $\C$-linear homogeneous symmetrizing form $\tr$ of degree $-2\ell(w_0)$ on $C$.
\end{corollary}

\begin{proof} We take the identity map as the anti-involution in the cellular structure. Then Part 1) of the corollary follows from Theorem \ref{borel} and the fact that $C$ is generated by its degree two and degree zero components. Part 2) of the corollary follows from the second statement in Theorem \ref{borel}.
\end{proof}

\begin{definition} Let $J\subset\Pi$ be a subset and $W_J$ the corresponding parabolic subgroup of $W$ generated by $s_\alpha$ for $\alpha\in J$. Let $W^J$ be the set of minimal length left coset representatives of $W_J$ in $W$. We define $$
C^J:=C^{W_J} .
$$
\end{definition}
In particular, for any $w\in W^J$ and $\alpha\in J$, we have $\ell(ws_\alpha)=\ell(w)+1$.

\begin{lemma}\text{(\cite{BGG0}, \cite[III, (4.2),(4.4)]{Hi})}\label{cj1} Let $J\subset\Pi$ be a subset. Then for any $w\in W^J$, $X_w\in C^J$. Moreover, the elements in the following subset $$
\{ X_w | w\in W^J \} $$ form a $\C$-basis of $C^J$.
\end{lemma}

\begin{corollary}\label{cellular2} The set $\{X_w|w\in W^J\}$ is a $\Z$-graded cellular basis of $C^J$ which makes $C^J$ into a $\Z$-graded cellular algebra in the sense of \cite{HuMathas:GradedCellular}. Moreover, each cell module of $C^J$ is one-dimensional.
\end{corollary}

For any $w\in W$, there is a unique decomposition $w=du$, where $d\in W^J$ and $u\in W_J$. Let $d_J$ be the unique element in $W^J$ such that $w_0\in d_{J}W_J$. Let $w_J$ be the unique longest element in $W_J$.

\begin{theorem}\label{tr} The map $X_d\mapsto\delta_{d,d_{J}}$, $\forall\,d\in W^J$, can be extended uniquely to a $\C$-linear homogeneous symmetrizing form $\tr_J$ of degree $-2\ell(w_0)+2\ell(w_J)$ on $C^J$.
\end{theorem}

\begin{proof} It suffices to show that $\tr_J$ is non-degenerate. Let $d\in W^J$. Applying \cite[Theorem 2.5.5]{BjB}, there exist elements $y_i\in W^J$, $\ell(y_i)=\ell(d)+i$ for $0\leq i\leq k$, such that
$d=y_0<y_1<\cdots<y_k=d_{J}$. By the definition of the Bruhat order, there exists $\beta_i\in\Phi^+$ such that $y_{i+1}=y_{i}s_{\beta_i}$ and $\ell(y_{i+1})=\ell(y_i)+1$ for each $0\leq i<k$. It is clear that $\beta_i\notin J$ as otherwise $s_{\beta_i}\in W_J\setminus\{e\}$ which contradicts to the fact that $y_{i+1}\in W^J$.

Since $\beta_1\notin J$, we can choose $\alpha_0\in\Pi\setminus J$ such that $\<\omega_{\alpha_0},\beta_1^\vee\>\neq 0$. Applying Theorem \ref{borel}, we can deduce that $X_{y_1}$ appears with positive coefficient in the expansion of $X_{s_{\alpha_0}}X_d$, and all the other $X_u$ appears with non-negative coefficient in expansion of $X_{s_{\alpha_0}}X_d$. Similarly, as $\beta_2\notin J$, we can apply Theorem \ref{borel} again to choose $\alpha_1\in\Pi\setminus J$ such that $\<\omega_{\alpha_1},\beta_2^\vee\>\neq 0$ and $X_{y_2}$ appears with positive coefficient in the expansion of $X_{s_{\alpha_1}}X_{y_1}$, and all the other $X_u$ appears with non-negative coefficient in expansion of $X_{s_{\alpha_1}}X_{y_1}$. Iteratively, we can get $\alpha_0,\alpha_1,\cdots,\alpha_{k-1}\in\Pi\setminus J$, such that for each $i$, $\<\omega_{\alpha_i},\beta_{i+1}^\vee\>\neq 0$ and $X_{y_{i+1}}$ appears with positive coefficient in the expansion of $X_{s_{\alpha_i}}X_{y_i}$, and  all the other $X_u$ appears with non-negative coefficient in expansion of $X_{s_{\alpha_i}}X_{y_i}$. Finally, we can deduce that $X_{d_{J}}$ appears with positive coefficient in $$
X_{s_{\alpha_{k-1}}}\cdots X_{s_{\alpha_{1}}}X_{s_{\alpha_{0}}}X_d .
$$
Note that $\alpha_i\in\Pi\setminus J$ implies that $X_{s_{\alpha_i}}\in W^J$. Moreover, as $d_J$ is the unique longest element in $W^J$, in this case we actually have that
$X_{s_{\alpha_{k-1}}}\cdots X_{s_{\alpha_{1}}}X_{s_{\alpha_{0}}}X_d=NX_{d_{J}}$ for some $0<N\in\N$. This proves that $\tr_J$ is nondegenerate and hence completes the proof of the theorem.
\end{proof}

%
%

We define \begin{equation}\label{rj}
\partial_J:=\Delta_{w_J},\quad R^J:=R^{W_J}=\{x\in R|w(x)=x,\forall\,w\in W_J\}.
\end{equation}

By \cite[\S24.3.1, Page 496, Line -11]{EMTW}, we know that $\partial_J(R)\subseteq R^J$.

\begin{lemma}\text{\rm (\cite[Lemma 24.35, Theorem 24.36]{EMTW})}\label{2adjR} 1) $R$ is a free $R^J$-module with a basis $\{\Delta_w(\mu_J)|w\in W_J\}$. Moreover, the pairing $(f,g):=\partial_J(fg)$ is a non-degenerate homogeneous bilinear form of degree $-2\ell(w_J)$.

2) The inclusion $\iota_J: R^J\hookrightarrow R$ and the map $\partial_J: R\rightarrow R^J$ equip $R$ with the structure a graded Frobenius algebra over $R^J$.

3) There are the following two adjoint pairs of graded functors: $$
\bigl(\ind_{R^J}^{R}(-), \res_{R^J}^{R}(-)\bigr),\quad \bigl((\res_{R^J}^{R}(-), \ind_{R^J}^{R}\<-2\ell(w_J)\>(-)\bigr) .
$$
\end{lemma}

\begin{proof} 1) and 2) follows from \cite[Lemma 24.35, Theorem 24.36]{EMTW}. It is well-known that the first pair of graded functors is an adjoint pair. It remains to consider the second pair of graded functors. Applying 1), we can find a homogeneous basis $\{b_w|w\in W_J\}$ which is dual to the basis $\{\Delta_w(\mu_J)|w\in W_J\}$ under $\partial_J$. We define the map $\delta: R\rightarrow R\otimes_{R^J}R$ which maps $1$ to
$\sum_{w\in W_J}\Delta_w(\mu_J)\otimes_{R^J}b_w$. Then the homogeneous map $\delta$ induces the second adjoint pair of graded functors.
\end{proof}

The following result, which plays a key role in the proof of the next proposition, was communicated to me by Dr. Huang Lin.

\begin{lemma}\label{linlem} Let $J\subset\Pi$ be a subset. Then the natural map $\pi: R^J\rightarrow C^J$ is surjective.
\end{lemma}
\begin{proof} Let $d\in W^J$. Recall that $D:=\frac{1}{|W|}\prod\limits_{\alpha\in \Phi^+}\alpha\in S(\mathfrak{h}^*)$. Let $\overline{D}$ be its image in $C$. By Lemma \ref{Dema}, we have that $$
\partial_J(\Delta_{w_Jd^{-1}w_0}\overline{D})=\Delta_{w_J}\Delta_{w_Jd^{-1}w_0}\overline{D}=
\Delta_{d^{-1}w_0}(\overline{D})=X_{d},
$$
which implies that $X_d$ is the natural image of $\partial_J(\Delta_{w_Jd^{-1}w_0}D)\in R^J$ in $C^J$. Applying Lemma \ref{cj1}, we can deduce that the natural map $\pi: R^J\rightarrow C^J$ is surjective.
\end{proof}

Let $\Phi_J$ be the sub-root system of $\Phi$ generated by $J$. Set $\Phi_J^+:=\Phi^+\cap\Phi_J$. We define $$
\mu_J:=\prod\limits_{\alpha\in \Phi_J^+}\alpha .
$$
The following proposition plays a key role in the study of $\Z$-graded translation functors in next section.

\begin{proposition}\label{keyprop1} Let $J\subset\Pi$ be a subset. Then $C$ is a free $C^J$-module with a basis $\{\Delta_w(\mu_J)+I_W|w\in W_J\}$. Moreover, The following are two pairs of adjoint graded functors: $$
\Bigl(\ind_{C^J}^C(-), \res_{C^J}^C(-)\Bigr),\quad \Bigl(\res_{C^J}^C(-), \ind_{C^J}^C(-)\<-2\ell(w_J)\>\Bigr) .
$$
\end{proposition}

\begin{proof} Let $\pi: R\twoheadrightarrow C$ be the natural surjection. By Lemma \ref{linlem}, $\pi$ restricts a surjection $\pi_J: R^J\twoheadrightarrow C^J$.
By \cite[Lemma 24.35]{EMTW}, $R$ is a free $R^J$-module with a basis $B_J:=\{\Delta_w(\mu_J)|w\in W_J\}$. Applying Lemma \ref{linlem}, we can deduce that $C$ is generated by the image of $B_J$ in $C$ as a $C^J$-module. On the other hand, since $\dim_\C C=|W|$, $\dim_\C C^J=|W^J|$ and $|W|=|W_J||W^J|$, it follows that the image of $B_J$ in $C$ must be a $C^J$-basis of $C$. In particular, $C$ is a free $C^J$-module.

It remains to prove that two pairs of graded functors are adjoint pairs. The first pair is clearly a pair of adjoint graded functors. We only need to consider the second pair of graded functors. Let $M\in C\lgmod, N\in C^J\lgmod$. We want to prove there is a graded isomorphism \begin{equation}\label{bi0}
\ho_{C^J}\bigl(\res^C_{C^J}M,N\bigr)\cong\ho_{C}\bigl(M,\ind_{C^J}^C N\<-2\ell(w_J)\>\bigr) .
\end{equation}

In view of Lemma \ref{2adjR}, it suffices to show that $C\otimes_{C^J}N\cong R\otimes_{R^J}N$ as left $R$-modules. There is a natural surjection $\pi_0$ from
$R\otimes_{R^J}N$ onto $C\otimes_{C^J}N$. Since the free $R^J$-module $R$ has rank $|W_J|$, which is the same as the rank of the free $C^J$-module $C$, it follows that two $\C$-linear spaces $R\otimes_{R^J}N$ and $C\otimes_{C^J}N$ have the same dimensions. Hence $\pi_0$ is an isomorphism. The proposition then follows.
\end{proof}

For any $M\in C^J\lgmod$, we define $M^*\in C^J\lgmod$ as follows: $M^*:=\Hom_{\C}(M,\C)$ and $(af)(x):=f(ax)$, $\forall\,a\in C^J, x\in M, f\in M^*$. This is well-defined because $C^J$ is commutative. As a $\C$-linear space,
$M^{\ast}_k:=\Hom_{\C}(M_{-k},\C)$, where $\C$ is regarded a $\Z$-graded space concentrated in degree $0$. We use $\ast$ to denote the duality functor on both the category $\O$ (denoted by $\vee$ in \cite[\S3.2]{Hum}) and $C^J\lgmod$.

\begin{corollary}\label{sd} Let $\lam$ be an integral dominant weight. For any $M\in\O_\lam$, we have $\mathbb{V}_\lam(M^*)\cong (\mathbb{V}_\lam M)^*$.
\end{corollary}

\begin{proof} Using the trace form $\tr$ introduced in Theorem \ref{tr}, the corollary follows from the same argument used in the proof of \cite[Lemma 8]{So} (where it deals with only the case when $\lam=0$).
\end{proof}


For each $\alpha\in\Phi$, let $\mathfrak{g}_\alpha$ be the root space of $\mathfrak{g}$ corresponding to $\alpha$. Let $\tau$ be the standard anti-involution of $\fg$ which exchanges $\mathfrak{g}_\alpha$ with $\mathfrak{g}_{-\alpha}$ for all $\alpha\in\Phi^+$ and fixes all $h\in\fh$. Then $\tau$ induces an equivalence between $\rmod A_\lam$ and $A_\lam\lmod$ and hence between $\rmod A_\lam$ and $\rmod A_\lam^{\rm{op}}$.
Recall that $w_{0,\lam}$ is the unique longest element in $W_\lam$.

\begin{lemma}\label{selfdualC} Let $\lam$ be an integral dominant weight and $x\in W^\lam$. Then with the natural $\Z$-grading (arising from the Koszul grading on $\O_\lam$) on $\mathbb{V}_\lam P(x\cdot\lam)$ one has that $(\mathbb{V}_\lam P(x\cdot\lam))\<-\ell(w_0)+\ell(w_{0,\lam})\>$ is a graded self-dual $C_\lam$-module.
\end{lemma}

\begin{proof} Forgetting the grading, we have a $(C_\lam,A_\lam)$-bimodule isomorphism $$
\bigl((\mathcal{F}(P(w_0\cdot\lam))^\tau\bigr)^*\cong\mathcal{F}(P(w_0\cdot\lam)) .
$$
Applying the hom functor $\Hom_{A_\lam}(\mathcal{F}(P(w_0\cdot\lam)),-)$ to $\mathcal{F}(P(x\cdot\lam))$, we get that
$$\begin{aligned}
&\quad\,\mathbb{V}_\lam P(x\cdot\lam)\\
&\cong\Hom_{A_\lam}\bigl(\mathcal{F}(P(w_0\cdot\lam)),\bigl((\mathcal{F}(P(x\cdot\lam))^\tau\bigr)^*\bigr)\\
&\cong
\Hom_{A_\lam}\bigl((\mathcal{F}(P(x\cdot\lam))^\tau,\mathcal{F}(P(w_0\cdot\lam))^\tau\bigr)\\
&=\Hom_{A_\lam}\bigl(\mathcal{F}(P(x\cdot\lam)),\mathcal{F}(P(w_0\cdot\lam)\bigr)\\
&\cong \Hom_{C_\lam}\bigl(\mathbb{V}P(x\cdot\lam),\mathbb{V}P(w_0\cdot\lam\bigr)=\Hom_{C_\lam}\bigl(\mathbb{V}P(x\cdot\lam),C_\lam\bigr)
\cong(\mathbb{V}P(x\cdot\lam))^*.\end{aligned}
$$
It follows that as a graded $A_\lam$-module, $$
(\mathbb{V}_\lam P(x\cdot\lam))^*\cong \mathbb{V}_\lam P(x\cdot\lam)\<-k\> .
$$
for some $k\in\Z$.

We define $$
m:=\max\Biggl\{d\geq 0\Biggm|\begin{matrix}\text{$f\in\mathbb{V}_\lam P(x\cdot\lam)=\Hom_{\O}(P(w_0\cdot\lam),P(x\cdot\lam))$}\\
 \text{is homogeneous, $d=\deg f$}\end{matrix}\Biggr\}.
$$
In fact, we know that $m=\ell(w_0)-\ell(w_{0,\lam})+\ell(x)$ from Corollary \ref{gradedlength2}, but we do not need this in the following proof.

Since the socle of $P(x\cdot\lam)$ is a direct sum of some copies of $L(w_0\cdot\lam)$. We can fix a homogeneous map $f\in\Hom_{\O}(P(w_0\cdot\lam),P(x\cdot\lam))$ of deg $m$. Then $$
f\in\Hom_{\O}(P(w_0\cdot\lam),P(x\cdot\lam))\cong\Hom_{C_\lam}(C_\lam,\mathbb{V}(P(x\cdot\lam))).
$$
It follows from Theorem \ref{tr} that $$0\neq f^*\in\Hom_{C_\lam}(\mathbb{V}(P(x\cdot\lam))\<-k\>,C_\lam\<-2(\ell(w_0)-\ell(w_{0,\lam}))\>).$$ Note that $\deg f^*=\deg f=m$.
It follows from the maximality of $m$ and the symmetry of the graded Cartan numbers that $k-2(\ell(w_0)-\ell(w_{0,\lam}))\geq 0$.

Swapping the position of $P(x\cdot\lam))$ and $P(w_0\cdot\lam)$, we can also show that $k-2(\ell(w_0)-\ell(w_{0,\lam}))\leq 0$. Thus we can deduce that
$k-2(\ell(w_0)-\ell(w_{0,\lam}))=0$ as required.
\end{proof}

Recall that the anti-involution $\tau$ induces an equivalence between $\rmod A_\lam$
and $A_\lam\lmod$ and hence between $\rmod A_\lam$ and $\rmod A_\lam^{\rm{op}}$. Since both $A_\lam$ and $A_\lam^{\rm{op}}$ are basic algebra, this gives rise to an isomorphism $\tau: A_\lam\cong A_\lam^{\rm{op}}$ with $\tau^2=\id$, and hence an anti-involution $\tau$ of $A_\lam$.

Following \cite[(6.1)]{Str}, we have $$
A_\lam  =\End_{\O}(P_\lam)\cong \End_{C_\lam}((\mathbb{V}_\lam P_\lam)^*)\cong\Bigl(\End_{\O}(P_\lam)\Bigr)^{\rm{op}}=
A_\lam^{\rm{op}}.
$$
As a result, we get a degree $0$ homogeneous $\Z$-graded anti-involution  of $A_\lam$, which will be still denoted by $\tau$.

For each $x\in W^\lam=W^J$, let $e_x$ be the projection from $\mathbb{V}_\lam P_\lam$ onto $\mathbb{V}_\lam P(x\cdot\lam)$. Then $e_x$ is a degree $0$ homogeneous primitive idempotent which corresponds to the indecomposable projective module $\mP(x\cdot\lam)$.

\begin{corollary}\label{idem0} The set $\{e_x|x\in W^\lam\}$ is a unique complete set of pairwise orthogonal degree $0$ primitive idempotents of $A_\lam$. In particular, in the Koszul dual picture of Theorem \ref{koszul}, $e_x$ corresponds to the projection from $\oplus_{x\in W^J}L(w_Jx^{-1}w_0\cdot 0)$ onto $L(w_Jx^{-1}w_0\cdot 0)$, where $J$ is as defined in (\ref{ij1}). Moreover, the $\Z$-graded anti-involution $\tau$ of $A_\lam$ coincides with the $\Z$-graded anti-involution on the Ext Yoneda algebra appeared in Theorem \ref{koszul} induced from the self-duality of each simple module $L(w_Jx^{-1}w_0\cdot 0)$ in $\O_0^\lam$.
\end{corollary}

\begin{proof} This follows from Corollary \ref{p2p}.
\end{proof}

\begin{definition}\label{duality0} Let $\circledast$ be the duality functor on $\rgmod A_\lam$ defined as follows: for any $M\in \rgmod A_\lam$, $M^{\circledast}=M^*=\Hom_{\C}(M,\C)$ as a $\C$-linear space, and $(fa)(x):=f(x\tau(a))$, $\forall\,f\in M^*, a\in A_\lam, x\in M$. As a $\C$-linear space,
$M^{\circledast}_k:=\Hom_{\C}(M_{-k},\C)$, where $\C$ is regarded a $\Z$-graded space concentrated in degree $0$.
Similarly, we can define a duality functor $\circledast$ on $A_\lam\lgmod$.
\end{definition}

\begin{remark} In \cite[\S6.1.2]{Str}, Stroppel defined a duality functor by $M^{\circledast}:=\Hom_{A_\lam}(M,A_\lam)$. However, this is incorrect because with this definition $L^{\circledast}=0$ whenever $L$ is a simple module which is not isomorphic to $\mathcal{F}(L(w_0\cdot\lam))$.
\end{remark}




\section{Graded lift of translation functors}

Translation functors are some special kind of projective functors and play some important role in the study of the BGG category $\O$. In this section, we shall
introduce and study the graded lift of translation functors, generalizing the earlier work \cite{Str} of Stroppel in the semiregular case to the general situation.

Let $\lam$ be an integral dominant weight. We use $\pr_\lam$ to denote the projection functor from $\O$ to $\O_\lam$. For each $\mu\in\Lam$, we use $\overline{\mu}$ to denote the unique weight in $W\mu\cap\Lam^+$. We consider the translation functor $\theta_\lam^{\rm{on}}: \O_0\rightarrow\O_\lam$, and its adjoint $\theta_\lam^{\rm{out}}:\O_\lam\rightarrow\O_0$. By definition, for any $M\in\O_0, N\in\O_\lam$, $$
\theta_\lam^{\rm{on}}(M):=\pr_\lam\bigl(L(\overline{\lam})\otimes M\bigr),\quad
\theta_\lam^{\rm{out}}(N):=\pr_0\bigl(L(\overline{-\lam})\otimes N\bigr),
$$

\begin{lemma}\text{(\cite[Lemma 2.5]{Bac},\cite[Chapetr 7]{Hum},\cite{J2},\cite{J3})}\label{lift0}  Let $\lam$ be an integral dominant weight and $x\in W^\lam$, \begin{enumerate}
\item[1)] both $\theta_\lam^{\rm{on}}$ and $\theta_\lam^{\rm{out}}$ are exact, commute with direct sum, duality functor and send projectives to projectives;
\item[2)] there are two adjoint pairs of functors: $(\theta_\lam^{\rm{on}},\theta_\lam^{\rm{out}})$, $(\theta_\lam^{\rm{out}},\theta_\lam^{\rm{on}})$;
\item[3)] $\theta_\lam^{\rm{on}}\bigl(L(xw_{0,\lam}\cdot 0)\bigr)\cong L(x\cdot\lam)$, $\theta_\lam^{\rm{on}}\bigl(\Delta(y\cdot 0)\bigr)\cong \Delta(y\cdot\lam)$ for any $y\in W$ and $\theta_\lam^{\rm{on}}\bigl(L(y\cdot 0)\bigr)=0$ whenever
$y\notin W^\lam w_{0,\lam}$;
\item[4)] $\theta_\lam^{\rm{out}}\bigl(P(x\cdot\lam)\bigr)\cong P(xw_{0,\lam}\cdot 0)$, $\theta_\lam^{\rm{out}}\bigl(L(x\cdot\lam)\bigr)$ is self-dual and has both the simple head and simple socle which is isomorphic to $L(xw_{0,\lam}\cdot 0)$.\end{enumerate}
\end{lemma}

In \cite[\S8]{Str}, Stroppel introduced certain $\Z$-graded lifts of $\theta_\lam^{\rm{on}}, \theta_\lam^{\rm{out}}$ in the semiregular case, i.e., when $W_\lam=\{1,s\}$ for some $s\in S$. The following lemma is a natural generalization of her result \cite[Theorems 8.1, 8.2]{Str}.

\begin{lemma}\label{lift1} Let $\lam$ be an integral dominant weight. The functor $\theta_\lam^{\rm{on}}$ is $\Z$-gradable with a $\Z$-graded lift: $$\begin{aligned}
\bbtheta_\lam^{\rm{on}}:&\,\, \rgmod A \rightarrow \rgmod A_\lam,\\
& M\mapsto M\otimes_{A}\Hom_{C_\lam}\Bigl(\mathbb{V}_\lam P_{\lam}, \res_{C_\lam}^{C}\mathbb{V}P\Bigr) ,
\end{aligned}
$$
while the functor $\theta_\lam^{\rm{out}}$ is $\Z$-gradable with a $\Z$-graded lift: $$\begin{aligned}
\bbtheta_\lam^{\rm{out}}:&\,\, \rgmod A_\lam \rightarrow \rgmod A,\\
& M\mapsto M\otimes_{A_\lam}\Hom_{C}(\mathbb{V}P,\ind_{C_\lam}^C\mathbb{V}_\lam P_{\lam}\<-\ell(w_{0,\lam})\>) .
\end{aligned}
$$
\end{lemma}

\begin{proof} This follows from \cite[Lemma 3.4]{Str} and \cite[2.2]{Bas}.
\end{proof}

The following result generalizes \cite[Theorem 8.4]{Str}, where Stroppel dealt with the semiregular case.

\begin{theorem}\label{adj2} Let $\lam$ be an integral dominant weight. Then there are the following two pairs of adjoint graded functors: $$
\bigl(\bbtheta_\lam^{\rm{on}},\bbtheta_\lam^{\rm{out}}\<-\ell(w_{0,\lam})\>\bigr),\quad
\bigl(\bbtheta_\lam^{\rm{out}},\bbtheta_\lam^{\rm{on}}\<\ell(w_{0,\lam})\>\bigr) .
$$
\end{theorem}

\begin{proof} For any $M\in\rgmod A$ and $N\in\rgmod A_\lam$, we have $$\begin{aligned}
\ho_{A_\lam}\bigl(\bbtheta_\lam^{\rm{on}}M,N\bigr)&=\ho_{A_\lam}\Bigl(M\otimes_A \ho_{C_\lam}\bigl(\mathbb{V}_\lam P_{\lam}, \res_{C_\lam}^{C}\mathbb{V}P\bigr),N\Bigr)\\
&\cong\ho_A\Bigl(M,\ho_{A_\lam}\bigl(\ho_{C_\lam}(\mathbb{V}_\lam P_{\lam}, \res_{C_\lam}^{C}\mathbb{V}P),N\bigr)\Bigr),\\
\ho_{A}\bigl(M,\bbtheta_\lam^{\rm{out}}\<-\ell(w_{0,\lam})\>N\bigr)&=\ho_{A}\Bigl(M, N\otimes_{A_\lam} \ho_{C}\bigl(\mathbb{V}P, \ind_{C_\lam}^C\mathbb{V}_\lam P_{\lam}\<-2\ell(w_{0,\lam})\>\bigr)\Bigr)\\
&\cong \ho_{A}\Bigl(M, N\otimes_{A_\lam} \ho_{C_\lam}\bigl(\res_{C_\lam}^C\mathbb{V}P, \mathbb{V}_\lam P_{\lam}\bigr)\Bigr),
\end{aligned}
$$
where the last isomorphism follows from Proposition \ref{keyprop1}.

Similarly, we have $$\begin{aligned}
\ho_{A_\lam}\bigl(\bbtheta_\lam^{\rm{out}}M,N\bigr)&=\ho_{A}\Bigl(M\otimes_{A_\lam} \ho_{C}\bigl(\mathbb{V}P, \ind_{C_\lam}^C\mathbb{V}_\lam P_{\lam}\<-\ell(w_{0,\lam})\>\bigr),N\Bigr)\\
&\cong \ho_{A}\Bigl(M, \ho_{A_\lam}\bigl(\ho_{C_\lam}\bigl(\res_{C_\lam}^C\mathbb{V}P, \mathbb{V}_\lam P_{\lam})\<-\ell(w_{0,\lam}\>,N\bigr)\Bigr),\\
\ho_{A}\bigl(M,\bbtheta_\lam^{\rm{on}}\<\ell(w_{0,\lam})\>N\bigr)&=\ho_{A}\Bigl(M, N\otimes_{A_\lam} \ho_{C_\lam}\bigl(\mathbb{V}_\lam P_{\lam}, \res_{C_\lam}^{C}\mathbb{V}P)\<\ell(w_{0,\lam}\>\bigr)\Bigr).
\end{aligned}
$$

Therefore, it suffices to show that there is a graded right $A$-module isomorphism:\begin{equation}\label{eqa1}
N\otimes_{A_\lam}\check{Q}\cong \Hom_{A_\lam}(Q,N),\,\, N\otimes_{A_\lam}{Q}\cong \Hom_{A_\lam}(\check{Q},N),
\end{equation}
where $Q:=\Hom_{C_\lam}(\mathbb{V}_\lam P_{\lam}, \res_{C_\lam}^{C}\mathbb{V}P)$, $\check{Q}:=\Hom_{C_\lam}\bigl(\res_{C_\lam}^C\mathbb{V}P, \mathbb{V}_\lam P_{\lam}\bigr)$, and the right $A$-module structure on $\Hom_{A_\lam}(Q,N)$ is induced from the left $A$-module structure on $Q$.

Applying \cite[Theorem 10]{So}, we have $\res_{C_\lam}^{C}\mathbb{V}P\cong\mathbb{V}_\lam\theta_\lam^{\rm{on}}P$. Therefore, as both $P_\lam, \theta_\lam^{\rm{on}}P$ are projective, $$
Q\cong \Hom_{C_\lam}(\mathbb{V}_\lam P_{\lam}, \mathbb{V}_\lam\theta_\lam^{\rm{on}}P)\cong\Hom_{\O_\lam}(P_\lam,\theta_\lam^{\rm{on}}P)\cong \mathcal{F}(\theta_\lam^{\rm{on}}P) .
$$
It follows that $N\otimes_{A_\lam}\check{Q}\cong \Hom_{A_\lam}(Q,N)$ as a $\C$-linear space. In particular, both $N\otimes_{A_\lam}\check{Q}$ and $\Hom_{A_\lam}(Q,N)$ have the same dimensions.

We construct a map $\phi: N\otimes_{A_\lam}\check{Q}\rightarrow\Hom_{A_\lam}(Q,N)$ as follows: for any $n\in N, g\in\check{Q}$, we define $$
\phi(n\otimes g): f\mapsto n(g\circ f),\quad \forall\, f\in Q,
$$
and extends its $\C$-linearly to a $\C$-linear map. Note that $g\circ f\in\End_{C_\lam}\bigl(\mathbb{V}_\lam P_\lam\bigr)\cong\End_{\O_\lam}(P_\lam)\cong A_\lam$, so $n(g\circ f)$ is a well-defined element in $N$.
It is easy to check that $\phi$ is a well-defined graded right $A$-module homomorphism. It remains to show that $\phi$ is injective.

As a left $A_\lam$-module, $$\begin{aligned}
\check{Q}&=\Hom_{C_\lam}\bigl(\res_{C_\lam}^C\mathbb{V}P, \mathbb{V}_\lam P_{\lam}\bigr)\cong\Hom_{C_\lam}\bigl(\mathbb{V}\theta_\lam^{\rm{on}}P, \mathbb{V}_\lam P_{\lam}\bigr)\\
&\cong\Hom_{\O_\lam}(\theta_\lam^{\rm{on}}P,P_\lam)\cong \mathcal{F}(\theta_\lam^{\rm{on}}P)^\tau,
\end{aligned}
$$
where $\mathcal{F}(\theta_\lam^{\rm{on}}P)^\tau$ means the left $A_\lam$-module obtained from twisting the right $A_\lam$-module $\mathcal{F}(\theta_\lam^{\rm{on}}P)$ via $\tau$.
In particular, it is projective. We have an isomorphism $\iota: \check{Q}\cong\oplus_{i=1}^b A_\lam e_i$, where for each $i$, $e_i$ is a primitive idempotent in $A_\lam$ which corresponds to a projection from $P_\lam$ onto $P(x_i\cdot\lam)$ for some $x_i\in W^\lam$.

Therefore, in order to show $\phi$ is injective, it suffices to show that for any $n\in N, g_i\in\check{Q}$ such that $g_i=\tau^{-1}(e_i)$, $n\otimes g_i\neq 0$, we have $\phi(n\otimes e_i)\neq 0$. Note that $n\otimes g_i\neq 0$ implies that $ne_i\neq 0$. By definition of $\tau$, $g_i$ corresponds to the projection from a direct summand $N_i\cong P(x_i\cdot\lam)$ of $\theta_\lam^{\rm{on}}P$ onto the direct summand $P(x_i\cdot\lam)$ of $P_\lam$.
Take $f_i\in Q$ to be the projection from $P_\lam$ onto the direct summand $N_i$ of $\theta_\lam^{\rm{on}}P$. Then it is clear that $g_i\circ f_i=e_i$. Thus $$
\phi(n\otimes g_i)(f_i)=n(g_i\circ f_i)=ne_i\neq 0,
$$
as required. This prove that $\phi$ is injective and hence the right $A$-module isomorphism $N\otimes_{A_\lam}\check{Q}\cong \Hom_{A_\lam}(Q,N)$ because both sides have the same dimensions. In a similar way, one can also prove that there is right $A$-module isomorphism $N\otimes_{A_\lam}{Q}\cong \Hom_{A_\lam}(\check{Q},N)$. This completes the proof of the theorem.
\end{proof}

\begin{remark} 1) Note that the proof of the above theorem is different with the proof of \cite[Theorem 8.4]{Str}. Our proof essentially boils down to Proposition \ref{keyprop1} which does not appear elsewhere (to the best of our knowledge). After a first version of this paper was completed and put on the arXiv, Stroppel informed us that one should replace the duality functor $\circledast$ used in her proof of \cite[Theorem 8.4]{Str} with the linear duality functor $\Hom_{\C}(-,\C)$ (which turns left modules into  right modules), and then it will not affect her other argument;

2) The second adjoint pair appeared (without proof) in \cite[(8)]{CM} with a reference to \cite[Theorem 38]{MOS}. However, one can check that the proof of \cite[Theorem 38]{MOS} relies on \cite[Theorem 8.4]{Str} which deals with only the semiregular case.
\end{remark}

The next theorem is also a generalization of \cite[Theorem 8.1]{Str} from semiregular case to the general case.

\begin{theorem}\label{4results} Let $\lam$ be an integral dominant weight and $x\in W^\lam$. We have $$\begin{aligned}
&\bbtheta_\lam^{\rm{on}}(\mD(x\cdot 0))\cong \mD(x\cdot\lam),\quad \bbtheta_\lam^{\rm{on}}(\mD(xw_{0,\lam}\cdot 0))\cong \mD(x\cdot\lam)\<-\ell(w_{0,\lam})\> ,\\
&\bbtheta_\lam^{\rm{on}}(\mL(xw_{0,\lam}\cdot 0))\cong\mL(x\cdot\lam)\<-\ell(w_{0,\lam})\>,\, \bbtheta_\lam^{\rm{out}}(\mP(x\cdot\lam))\cong\mP(xw_{0,\lam}\cdot 0).
\end{aligned}
$$
\end{theorem}

\begin{proof} Forgetting the grading, the statement of the theorem all hold by Lemmas \ref{lift0}, \ref{lift1}. It remains to determine the shifts.

Using Theorem \ref{2cases} and the fact that $P_{y,w_0}=1$ for any $y\in W$, we know that for any $y\in W$, $\mathbb{V}(\Delta(y\cdot\lam)$ is one-dimensional and hence concentrates in one degree.
Applying Corollary \ref{gradedlength1} we see that both $\mathbb{V}\mD(x\cdot 0)$ and $\mathbb{V}\mD(x\cdot\lam)$ concentrate in degree $\ell(w_0)-\ell(x)$. Hence
$\mathbb{V}\bbtheta_\lam^{\rm{on}}(\mD(x\cdot 0))\cong \mathbb{V}\mD(x\cdot\lam)$, and hence $\bbtheta_\lam^{\rm{on}}(\mD(x\cdot 0))\cong \mD(x\cdot\lam)$.
By similar argument, we can deduce that $\bbtheta_\lam^{\rm{on}}(\mD(xw_{0,\lam}\cdot 0))\cong \mD(x\cdot\lam)\<-\ell(w_{0,\lam})\> $.

As a result, we can also deduce the third isomorphism because $\mL(xw_{0,\lam}\cdot 0)$ (resp., $\mL(x\cdot\lam)$) is the unique simple head of $\mD(xw_{0,\lam}\cdot 0)$ (resp., $\mD(x\cdot\lam)$) and the statement holds upon forgetting the gradings.

Finally, all the remaining isomorphism follow from the first one, Theorem \ref{lift0} as well as Theorem \ref{adj2}.
\end{proof}

\begin{corollary} Let $\lam$ be an integral dominant weight. Then $\bbtheta_\lam^{\rm{out}}\bbtheta_\lam^{\rm{on}}\cong \bbtheta_{w_{0,\lam}}$.
\end{corollary}

\begin{proof} By definition, $\theta_\lam^{\rm{out}}\theta_\lam^{\rm{on}}$ is a projective functor. Since $\lam$ is dominant, $\Delta(\lam)=P(\lam)$. We know that $$
\theta_\lam^{\rm{out}}\theta_\lam^{\rm{on}}\Delta(0)\cong\theta_\lam^{\rm{out}}\Delta(\lam)
=\theta_\lam^{\rm{out}}P(\lam)\cong P(w_{0,\lam}\cdot 0)\cong\theta_{w_{0,\lam}}P(0).
$$
Hence $\theta_\lam^{\rm{out}}\theta_\lam^{\rm{on}}\cong \theta_{w_{0,\lam}}$ by the classification of indecomposable projective functors. It remains to determine the shifts. But this follows from Theorem \ref{4results} by considering their action on $\mD(0)$.
\end{proof}

\end{document}